\documentclass{siamonline181217}
\usepackage[utf8]{inputenc}
\usepackage{fullpage}

\usepackage{enumitem}
\setlist[enumerate]{leftmargin=.5in}
\setlist[itemize]{leftmargin=.5in}

\usepackage{amsmath,amssymb,amscd,amsfonts}
\usepackage{graphicx}
\usepackage{hyperref}
\usepackage{pstricks}
\usepackage[section]{placeins}
\usepackage{latexsym}
\usepackage{epsfig}
\usepackage{ulem} 

\usepackage{listings}
\usepackage{lstautogobble}
\DeclareMathOperator{\conv}{conv}

\begin{document}

\title{Stochastic Tverberg theorems and their applications in multi-class logistic regression, data separability, and centerpoints of data}
\author{J.A. De Loera and T.A. Hogan}

\maketitle

\begin{abstract} We present new stochastic geometry theorems that give bounds on the probability that 
$m$ random data classes all contain a point in common in their convex hulls.  
We apply these stochastic separation theorems to obtain bounds on the probability of existence of maximum likelihood estimators in multinomial logistic regression. We also discuss connections to condition numbers for analysis of steepest descent algorithms in logistic regression and to the computation of centerpoints of data clouds.
\end{abstract}


\newtheorem*{theoremp}{Theorem}
\newcommand{\ba}{\boldsymbol{a}}
\newcommand{\bb}{\boldsymbol{\beta}}
\newcommand{\bc}{\boldsymbol{c}}
\newcommand{\bd}{\boldsymbol{d}}
\newcommand{\w}{\boldsymbol{w}}
\newcommand{\p}{\boldsymbol{p}}
\newcommand{\x}{\boldsymbol{x}}
\newcommand{\y}{\boldsymbol{y}}
\newcommand{\q}{\boldsymbol{q}}
\newcommand{\z}{\boldsymbol{z}}
\newcommand{\cc}{\boldsymbol{c}}
\newcommand{\R}{\mathbb{R}}
\newcommand{\CA}{\mathcal{C}}
\newcommand{\Z}{\mathbb{Z}}
\newcommand{\lip}{\text{Lip}}
\newcommand{\ff}{\mathcal{F}}
\newcommand{\X}{\boldsymbol{X}}

\newcommand{\Ne}{\mathbb{Ne}}

\newcommand{\pr}{\mathbb{P}}

\section{Introduction}

This paper shows how methods from stochastic convex geometry can be successfully used in the foundations of data science. 
Before we discuss the geometric results, we discuss their implications:

Logistic regression is perhaps the most widely used non-linear model in multivariate statistics and supervised learning~\cite{GenLinMod}. Statistical inference for this model relies on the theory of maximum likelihood estimation. In the binary classification case, given $n$ independent observations $(\x_i, y_i), i  = 1, \dots, n$, logistic regression links the response $y_i \in \{-1,1\}, i = 1, \dots, n$ to the 
covariates $\x_i \in \mathbb{R}^d$ via the logistic model 
$$ \pr( y_i = 1 | \x_i) = \sigma(\x_i'\bb) ,\hspace{.5cm} \sigma(t) := \frac{e^t}{1 + e^t};$$
here $\bb \in \R^p$ is the unknown vector of regression coefficients. In this model, the \emph{log-likelihood} is given by
$$l(\bb) = \sum_{i = 1}^{n} - \log(1 + \exp(-y_i\x_i'\bb))$$
and, by definition, the \emph{maximum likelihood estimate} (MLE) is any maximizer of this functional. The basic intuition behind 
this method is as follows: we seek coefficients $\bb$ so that $\pr( y_i = 1 | \x_i)$ corresponds as closely as possible with the observations $(\x_i, y_i)$.   For example, if $\x_i'\bb$ and $y_i$ have different signs, there is a larger ``penalty" expressed in the log-likelihood, since in that case $\exp(-y_i\x_i'\bb) > 1$. See \cite{ISLR} for further discussion and examples.

 One difficulty arising in machine learning is that the MLE does not exist in all situations. In fact, given two data classes, say one of red points (where $y_i = 1$), and one of blue points (where $y_i = -1$), it is well-known that an MLE exists if and only if the convex hulls of the blue points intersects the convex hull of the red points \cite{albertanderson,silvapulle}. Although an appealing criterion for existence, this geometric characterization leads to another question: {\it How much training data do we need, as a function of the dimension of the covariates of the data, before we expect an MLE to exist with high probability?}

The seminal work of Cover~\cite{cover2color} (adapting a technique originally due to Schl\"afli~\cite{schlafli}) provides an answer in a special case. When applied to logistic regression, Cover's main result states the following: assume that the $\x_i$'s are drawn i.i.d. from a continuous probability distribution $F$ and that the class labels are independent from $\x_i$, and have equal marginal probabilities; i.e., 
$\pr(y_i = 1| \x_i ) = 1/2$. Then Cover showed that as $d$ and $n$ grow large in such a way that $d/n \rightarrow k$, the convex hulls of the data points asymptotically overlap - with probability tending to one - if $k < 1/2$, whereas they are separated - also with probability tending to one - if $k > 1/2$.  When the class labels are not independent from the $\x_i$, the problem is more difficult. In this case, Cand\`es and Sur~\cite{candessur} proved that a similar phase transition occurs, and is parameterized by two scalars measuring the overall magnitude of the unknown sequence of regression coefficients.

 Tukey introduced a notion of depth for a point $p$ relative to a data cloud $S$, as the smallest number of data points in a closed half-space with boundary through $p$ (see \cite{tukey1975, Rousseeuw1998} and references therein). We say a point $p$ has \emph{half-space depth $k$ in $S$} if that every half-space containing $p$ contains at least $k$ points in $S$.  
 A \emph{centerpoint} of an $n$ point data set $S \subset \R^d$ is a point $p$ such that every half-space containing $p$ has at least 
$\frac{n}{d+1}$ points in $S$, thus it is a point of depth at least $\frac{n}{d+1}$.  In a way a centerpoint is a generalization of the notion of median for high-dimensional data. Centerpoints are useful in a variety of applications (see e.g., \cite{jesus+xavier+frederic+nabil} for references). Unfortunately, obtaining a centerpoint is difficult, and the current best randomized algorithm constructs a centerpoint in time $O(n^{d+1} + n\log n)$\cite{chan, Miller+Sheeny2010}. Thus finding an approximate centerpoint of a set is of interest.


\section*{Consequences of our geometric results}~\label{discuss}

The first contribution of our paper is to further develop the connection between  geometric probability (Cover's result), discrete geometry (Tverberg-type results), and the conditions for the existence of MLEs.  Our paper discusses the generalization of Cover's stochastic separation problem to more than two colors by studying so-called \emph{Tverberg partitions}- a partition of a data set into classes so that the intersection of all the convex hulls of the classes is nonempty. 


Each of our stochastic-geometric theorems has a nice implication. Table~\ref{MLEexist} summarizes our theorems (middle column) as  well as their consequences to the existence of the Maximum-likelihood estimator in terms of the size of the data set (right column). 

\begin{table}[ht!]
  \begin{center}
    \begin{tabular}{l|c|c} 
      \textbf{Deterministic version} & \textbf{Stochastic version} & \textbf{Likely MLE Existence}\\

      \hline
      Radon & Cover's Theorem\cite{cover2color} & pair of data classes (mentioned above)\\
      \hline
      Tverberg & Thms~\ref{probtverberg},\ref{TTT} & all pairs of data classes \\ & & (Theorem \ref{mleexistthm} part 1.)\\
      \hline
      Radon with tolerance & Thm~\ref{stoctolradon}& pair of data classes with outliers removed\\ & & (Theorem \ref{mleexistthm} part 2.)\\
      \hline
      Tverberg with tolerance & Thms~\ref{probtverbergtolrand},\ref{probtverbergtol},\cite{soberon_2018} & all pairs of data classes with outliers removed\\ & & (Theorem \ref{mleexistthm} part 2.)
    \end{tabular}
    \vspace{.5cm}
    \caption{Stochastic analogues of Tverberg's theorem and their implications for existence of MLEs. By ``Likely MLE Existence", we mean that one can bound below the probability of MLE existence as a function of the number of input data points, according to the corresponding theorems in the ``Stochastic" column.}
    
    \label{MLEexist}
  \end{center}
\end{table}

%

There are two common approaches to extend binary classification to multi-class classification: \emph{``one-vs-rest"} and \emph{``one-vs-one"}. Suppose the data has $C > 2$ classes. In ``one-vs-rest", we train $C$ separate binary classification models. Each classifier $f_c$ for $c \in \{1,\dots, C\}$ is trained to determine whether or not an example is part of class $c$ or not. To predict the class for a new sample $\x$, we run all $C$ classifiers on $\x$ and choose the class with the highest score: $\hat{y} = \arg \max_{c \in \{1, \dots, C\} } f_c(\x).$  In ``one-vs-one" regression, we train $\binom{C}{2}$
 separate binary classification models, one for each possible pair of classes. To predict the class for a new sample $\x$
, we run all $\binom{C}{2}$  classifiers on  $\x$ and choose the class with the most votes.
 
To apply ``one-vs-one" multinomial logistic regression, we would like to ensure that the MLE exists between the data corresponding to every pair of labels. The next theorem applies our stochastic Tverberg theorem to give a sufficient condition for all these MLEs to exist with high probability (a sequence of events $X_n, n \geq 1$, occurs with \emph{with high probability} if 
$\lim_{n \rightarrow \infty} P(X_n) = 1$.)

\begin{theorem}[Stochastic Tverberg theorems applied to multinomial regression]
Fix $\epsilon > 0$.
Assume that the $\x_i$'s are drawn i.i.d. from a centrally symmetric continuous probability distribution $F$ on $\mathbb{R}^d$ and that the class labels are independent from $\x_i$, and have equal marginal probabilities; i.e., 
$\pr(y_i = k| \x_i ) = 1/m$ for all $k \in \{1, \dots, m\}$.\\
Then
\begin{enumerate}
\item Letting the number of data points $f(m)$ grow as a function of the number of labels $m$, the MLE exists between the data corresponding to every pair of labels with high probability as long as
$$f(m) \gg (1 + \epsilon)  m \log_2(m) \ln(\ln(m)).$$

\item Suppose the number of data points is $g(t)$, where we fix $m$ the number of labels, and $g(t)$ is a function of $t$- the number of outliers to be removed from the data set. Then the MLE exists between the data corresponding to every pair of labels with high probability if any $t$ points are removed, so long as $t \ll g(t)/m$.
\label{mleexistthm}
\end{enumerate} 
\end{theorem}

The same bound applies to ``one-vs-rest" logistic regression, since  MLE existence in that case is a weaker condition. The various special cases of Stochastic Tverberg theorems are thus useful in different kinds of classification problems, and these observations are summarized in Table~\ref{MLEexist}.

The last two rows of the table of summarizing our results was motivated by the challenge of dealing with outlier data and seeking robust classification of data, we rely on an additional parameter: \emph{tolerance}. A $t$-tolerant partition, which will be defined formally later, 
but it is a notion of ``robust"  intersection in the sense that the intersection of the convex hulls of the subsets remains non-empty even  
after any  $t$ points are removed. See Figure~\ref{tolex} for an example of a $1$-tolerant partition. 

\begin{figure}[hb]
\center{
\includegraphics[scale=1]{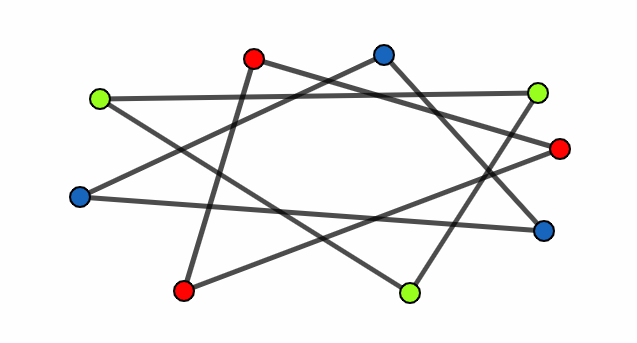}}

\caption{A partition in three data classes with tolerance one. All three convex hulls intersect even after any one point is removed.}
\label{tolex}
\end{figure}
 
The parameter of tolerance is also significant in studying MLE existence. A natural observation is that $t$ tolerant partitions correspond to robust MLE existence. Any $t$ points, possibly corrupted or outlier data, can be removed and still the convex hulls of the data with each label intersect.

In fact, the parameter of tolerance is also similar to an important parameter used to guarantee to the convergence speed of first order methods for finding MLEs. Recently, when studying binomial logistic regression,  Freund, Grigas and Mazumunder~\cite{grigas} introduced the following notion to quantify the extent that a dataset is non-separable (where $a^- := -\min\{a,0\}$ denotes the negative part of $a$):


$$ \text{ DegNSEP*} := \min_{\bb \in \mathbb{R}^p} \frac{1}{n} \sum_{i=1}^n [y_i \bb^T \x_i]^-$$
$$\text{ s.t. } \| \bb \| =1.$$

DegNSEP*
is thus the smallest (over all normalized models $\bb$) average misclassification
error of the model $\bb$ over the $n$ observations.
They showed that the condition number DegNSEP* informs the computational properties and guarantees of the standard deterministic first-order steepest descent solution method for logistic regression.
Let us now briefly discuss how the parameter of tolerance for Radon partitions (Tverberg 2-partitions) can be viewed as a discrete analogue of DegNSEP*. 

Define PertSEP* as the smallest (or more precisely, the infimum thereof) perturbation $\Delta \X$ of the feature data $\X$ which will render the perturbed problem instance $(\X + \Delta \X, y)$ separable. Namely, 
$$\text{PertSEP*} := \inf_{\Delta \X} \frac{1}{n} \|\Delta\|_{\cdot,1}$$
$$\text{ s.t. } (\X + \Delta \X, y) \text{ is separable.}$$
In Proposition 2.4 of \cite{grigas} it is shown that DegNSEP* = PertNSEP*.

In this paper we introduce a new parameter $\text{PertSEP*}_0$ simply defined as the $L_0$ 
norm of the smallest perturbation of the feature data $\X$ which will render the perturbed problem 
instance $(\X + \Delta \X, y)$ separable. In other words, it is the minimal number of data points we could move to make the data set separable, normalized by the total number of data points. Namely, 
$$\text{PertSEP*}_0 := \inf_{\Delta \X} \frac{1}{n} \|\Delta\|_{\cdot,0}$$
$$\text{ s.t. } (\X + \Delta \X, y) \text{ is separable.}$$

The following theorem shows that the tolerance of a Radon partition is given by $\text{PertSEP*}_0$:
\begin{theorem}
Suppose that $\X = \X_1 \cup \X_2$, $|\X| = n$ is a Radon partition with tolerance precisely equal to $t$. Then viewing $\X$ as a labeled dataset (with $(\X, y) = \{(\x,-1) : \x \in \X_1 \} \cup \{(\x,1,): \x \in \X_2\}$), we have that 
$$\text{PertSEP*}_0 = t/n.$$
\label{discretesep}
\end{theorem}

Theorem~\ref{discretesep}, combined with a result of Sober\'on, has a corollary, stated precisely in the next section, which roughly says that $\text{PertSEP*}_0$ of a randomly bi-partitioned point set asymptotically approaches $1/2$. This is the highest possible value one could hope for since, by definition, $\text{PertSEP*}_0$ of \emph{any} two class data set is bounded above by $1/2$. In fact, this result extends easily to the multi-class setting. In other words, for a large randomly $m$-partitioned data set, we expect $\text{PertSEP*}_0$ of every pair of data classes to be close to $1/2$ - independent of both the dimension of the covariates, as well as the number of classes $m$.

For further discussion of $\text{PertSEP*}$ and $\text{DegNSEP}*$ for two-class data, including more probabilistic aspects of these condition numbers and many interesting implications for steepest descent algorithms, see \cite{grigas}.

We also discuss how our geometric probability results are related to the problem of computing approximations to centerpoints of datasets. Table \ref{tab:table2} and the discussion that follows summarize our contributions.
Tverberg's theorem implies that every data set has a centerpoint, as the Tverberg intersection point of  a Tverberg partition must be a point of half-space depth one in each of the $m = \lceil \frac{n}{d+1}\rceil$ color classes. Hence an effective version is desirable as a method to obtain centerpoints. The proof of Radon's lemma is constructive and, in fact, one of the most notable randomized algorithms for computing approximate centerpoints works by repeatedly replacing subsets of $d + 2$ points by their Radon point. In contrast, no known polynomial time algorithm exists for computing exact Tverberg points. Thus, fast algorithms for approximate Tverberg points have been introduced in \cite{CEMST96,MW13,RS16}. If one is interested in probabilistic algorithms for finding Tverberg partitions, the main results of our paper 
can be used to give  expected performance of algorithms where we obtain Tverberg partitions by random choice, so long as the points come from a 
 balanced distribution.

In particular, our Theorem~\ref{TTT} suggests a trivial algorithm for finding a Tverberg partition among a set of i.i.d. points drawn from a distribution which is balanced about a point $\p$. According to Theorem~\ref{TTT}, a random equipartition of $m$ such points into less than $m/log_2(m)$ sets should produce a Tverberg partition with high probability. This trivial non-deterministic algorithm was also suggested by Sober\'on, except using a random allocation rather than equi-partition. Our asymptotic results improve the bounds on expected performance of Sober\'on's proposed algorithm (random allocation) for points from a balanced distribution as well. We summarize the performance and time complexity of various algorithms for obtaining Tverberg partitions, including our own (last two rows),  in Table \ref{tab:table2}.
\begin{table}[h!]
  \begin{center}
    \begin{tabular}{l|c|r} 
      \textbf{Method} & \textbf{Number of Colors} & \textbf{Time complexity}\\

      \hline
      Tverberg & $\lfloor (m+1)/(d+1) \rfloor$ & PPAD (unknown if polynomial)\\
      \hline
      Mulzer, Werner~\cite{MW13}&$m / (4d+1)^3 $& $d^{O(\log{d})}m$\\
      \hline
      
      Rolnick, Sober\'on ~\cite{RS16}& $m / d(d+1)^2 $ with error prob. $\epsilon$ & weakly poly. in $m$,$d$ and $\log(1/\epsilon)$\\
      \hline
      Random equi-partition & $O(\frac{m}{\log_2(m)})$& $O(m)$\\
      \hline
      Random Allocation &  $O(\frac{m}{\log_2(m)(\ln(\ln(m)))})$& $O(m)$
    \end{tabular}
    
    \caption{Approximate Tverberg Partitions for balanced distributions.}
    
    \label{tab:table2}
  \end{center}
\end{table}

%

Section~\ref{TTypeThms} presents our geometric tools and results, and Section~\ref{bounds} contains the proofs of our new results. 

\section{Our geometric methods: Stochastic Tverberg-type theorems}~\label{TTypeThms}

We begin by remembering Tverberg's celebrated theorem~\cite{Tv} which generalizes Radon's lemma to $m$-partitions (see \cite{barany+soberonsurvey, jesus+xavier+frederic+nabil} for references and the importance of this theorem):

\begin{theoremp}[Theorem: (H. Tverberg 1966)] Every set $S$ with at least $(d+1)(m-1)+1$ points in Euclidean 
$d$-space has at least one $m$-Tverberg partition (with tolerance zero). 
\end{theoremp}
%

The notion of ``tolerant Tverberg theorems" was pioneered by Larman~\cite{Larman:1972tn} and refined over the years, such as in the following result due to Sober\'on and Strausz~\cite{sobstr}. Here is the precise definition:

\begin{definition}
Given a set $S \subset \R^d$, a \emph{Tverberg $m$-partition of $S$ with tolerance $t$} is a partition of $S$ into $m$ subsets $S_1, \dots S_m$  with the property that all $m$ convex hulls of the $S_i$ intersect after any $t$-points are removed. 
In other words, for all $\{\x_1, \dots, \x_t\} \in S$,  we have 
$$\bigcap_{i \in [m]} \conv(S_i \setminus \{\x_1, \dots, \x_t\}) \neq \varnothing.$$
\end{definition}

%
%
%

\begin{theoremp}[Theorem: (Sober\'on, Strausz 2012)]
Every set $S$ with at least $(t+1)(m-1)(d+1) + 1$ points in $\R^d$ has at least one Tverberg $m$-partition with tolerance $t$. 
In other words, $S$ can be partitioned into $m$ parts $S_1, \dots, S_m$ so that for all $\{\x_1, \dots, \x_t\} \in S$,  we have 
$$\bigcap_{i \in [m]} \conv(S_i \setminus \{\x_1, \dots, \x_t\}) \neq \varnothing.$$
\end{theoremp}
More recently, P. Sober\'on proved the following bound~\cite{soberon_2018}. Let $N$ denote the smallest positive integer such that a Tverberg $m$-partition with tolerance $t$ exists among any $N$ points in dimension $d$. Then $N = mt+ O(\sqrt{t})$ for fixed $m$ and $d$. The proof of this result relies on the probabilistic method and, as Sober\'on remarked, can in fact be used to prove a Stochastic Tverberg-type result, which we will revisit later.

\subsection*{Prior Stochastic Tverberg theorems}

Before stating our main results, we introduce two models for random partitioned data point sets. In both models will use the term colors instead of subsets, for ease of notation.
Hereafter, when we refer to a continuous distribution on $\R^d$, we mean continuous with respect to the Lebesgue measure on $\R^d$. We defer proofs of the new results stated until the next section.

Our first model is a so-called \emph{random equi-partition model} i.e., we ensure that every color has the same number of points. More specifically, given integers $m$ and $n$ and a continuous probability distribution $D$ on $\mathbb{R}^d$, we let $\mathcal{E}_{m,n,D}$ denote a random equi-partitioned point set with $mn$ points, consisting of $m$ colors, and $n$ points of each color, distributed independently according to $D$.

Our second model is a \emph{random allocation model}: Given integers $k$ and $m$ and a continuous probability distribution $D$ on $\mathbb{R}^d$, we let $\mathcal{R}_{m,k,D}$ denote a random point set with $k$ points i.i.d. according to $D$, which are randomly colored one of $m$ colors with uniform probability ($1/m$ for each color).


For example, using these models we can state Cover's result as follows:
\begin{theoremp}[Theorem: (T. Cover 1965)]
If $D$ is a continuous probability distribution on $\R^d$, then
$$ \pr(\mathcal{R}_{2,n,D} \text{ is Radon}) = 1 - 2^{-n+1}\sum_{k = 0}^{d}\binom{n-1}{k}.$$

In particular, we have $$ \pr(\mathcal{R}_{2,2(d+1),D} \text{ is Radon}) = 1/2.$$ Furthermore, for any $\epsilon > 0$ and any sequence of continuous probability distributions $\{D_i\}, i\in \mathbb{Z}_+$ where each $D_d$ is a distribution on $\R^d$, we have
$$\lim_{i \rightarrow \infty} \pr(\mathcal{R}_{2,(1+\epsilon)2i,D_i} \text{ is Radon}) = 1$$
and
$$\lim_{i \rightarrow \infty} \pr(\mathcal{R}_{2,(1-\epsilon)2i,D_i} \text{ is Radon}) = 0$$
\end{theoremp}

To the best of the authors' knowledge, the first generalization of Cover's 1964 result to more than two colors appeared only recently in Sober\'on's paper~\cite{soberon_2018}:

\begin{theoremp}[Theorem: P. Sober\'on 2018]
Let $N, t, d, m$ be positive integers and let $\epsilon > 0$ be a real number. Given $N$ points in $\R^d$, a random allocation of them into $m$ parts is a Tverberg partition with tolerance $t$ with probability at least $1-\epsilon$, as long as 
$$ t + 1 \leq N/m - \sqrt{\frac{1}{2}
\left[(d+1)(m-1)N\ln(Nm)+N\ln\left(\frac{1}{\epsilon}\right)\right]}.$$
\label{Pablo}
\end{theoremp}

This result is quite remarkable. For any fixed $m,d,$ and $\delta$, it shows that the probability of a random allocation of  of $N$ points in $\R^d$ in $m$ colors having tolerance at least $(1-\delta)N/m$ approaches one as $N$ goes to infinity.
On the other hand, by pigeonhole principle, any allocation of $N$ points into $m$ colors must have one color with at most $N/m$ points. Thus, for a fixed number of colors $m$, the tolerance of a random partition is asymptotically as high as it could possibly be! By Theorem~\ref{discretesep}, this result yields the following corollary. 

\begin{corollary}\label{expectedpert}\sloppy For any sequence $\{\mathcal{R}_{(2,k,D)}\}, k \in \mathbb{N}$ of partitioned point sets with $D$ a distribution on $\R^d$, and any $\epsilon > 0$, we have $| \text{PertSEP*}_0(\mathcal{R}_{(2,k,d)}) - 1/2 |< \epsilon$ with high probability.

\end{corollary}

In fact, for fixed $d$ and $m$, Corollary \ref{expectedpert} can be extended to the multi-class setting.
In other words, for a large randomly $m$-partitioned data set, we expect $\text{PertSEP*}_0$ of every pair of data points to be close to $1/2$ :
\begin{theorem} Fix $\epsilon > 0$. For any distribution $D$ on $\R^d$ and any sequence $\{\mathcal{R}_{(m,k,D)}\},k \in \mathbb{N}$ of $m$-partitioned point sets $\mathcal{R}_{(m,k,D)} = \{X_1, \dots,  X_m\}$
    
    we have $$\lim_{k \rightarrow \infty} \left(\min_{X_i,X_j \in \mathcal{R}_{(m,k,D)}} \text{PertSEP*}_0(X_i \cup X_j) = 1/2 \right)$$ with high probability.
    \label{pertsepmulti}
    \end{theorem}

\subsection*{Our new stochastic geometric theorems}

Our first theorem is a geometric probability result similar to Sober\'on's and Cover's. It yields a Stochastic Tverberg theorem for equi-partitions (without tolerance). 

\begin{theorem}[Stochastic Tverberg theorem for equi-partitions]~\label{probtverberg}
Suppose $D$ is  a probability distribution on $\mathbb{R}^d$ that is balanced about some point $\p \in \mathbb{R}^d$, in the sense that every hyperplane through $\p$ partitions $D$ into two sets of equal measure. Then

$$ \left(1-\left(\frac{1}{2^{n-1}} \sum_{k = 0}^{d-1} \binom{n-1}{k} \right)\right)^m \leq\pr(\mathcal{E}_{m,n,D} \text{ is Tverberg })\leq 
\left( 2(1-2^{-n})^m - (1-2^{-n + 1})^m\right)^d.$$
\end{theorem}

In fact, the previous theorem is asymptotically tight in the number of colors $m$. This is shown by our next theorem, which establishes an interesting threshold phenomenon for Tverberg partitions.

\begin{theorem}[Tverberg Threshold Phenomena for equi-partitions]~\label{TTT}
Let $D$ be a continous probability distribution in $\mathbb{R}^d$ balanced about some point $\p \in \mathbb{R}^d$. Consider the sequence of random equi-partitioned point sets $\mathcal{E}_{m,f(m),D}$, where $m \in \mathbb{N}$, and $n = f(m)$ depends on $m$. 
Then $\mathcal{E}_{m,f(m),D}$ is Tverberg with high probability if $f(m) \gg \log_2(m)$, and 
$\mathcal{E}_{m,f(m),D}$ is not Tverberg with high probability if $f(m) \ll \log_2(m)$.

\end{theorem}

{\bf Remark:} It is also interesting to consider the same problem from the ``box convexity" setting where the convex hull of a set of points is defined to be the smallest box (with sides parallel to the coordinate axes) enclosing those points. Since checking convex hull membership is easier in the box convexity setting, this set up may be more relevant in certain applications. Our method of proof of Theorem~\ref{probtverberg} also works in box convexity setting, and we obtain the same bounds. 

We note that the number of points needed to reach the conclusion in Theorem~\ref{TTT} is independent of the dimension, as in the aforementioned result of Sober\'on \cite{soberon_2018}.

The next two theorems adapt both Cover's result and Theorem~\ref{probtverberg} to the setting of tolerance.

\begin{theorem}[Stochastic Tverberg with tolerance for equi-partition]~\label{probtverbergtol}
Suppose $D$ is  a probability distribution on $\mathbb{R}^d$ that is balanced about some point $\p \in \mathbb{R}^d$.

$$ \pr(\mathcal{E}_{m,n,D} \text{ is Tverberg with tolerance $t$})
\geq \left(1-2^{-\lfloor n / 2d \rfloor} \sum_{i = 1}^t \binom{\lfloor n / 2d \rfloor}{i}\right)^m$$

\end{theorem}

For the case of random bi-partitions, we can adapt Cover's result to obtain a Stochastic Radon theorem with Tolerance.

\begin{theorem}[Stochastic Radon with tolerance for random allocation]~\label{stoctolradon}
If $D$ is a continuous probability distribution on $\R^d$, then
$$ \pr(\mathcal{R}_{2,k,D} \text{ is Radon with tolerance $t$}) \geq 1- \left(2^{-\lfloor k/(2d+2) \rfloor} \sum_{i = 0}^{t} \binom{\lfloor k/(2d + 2) \rfloor}{i} \right).$$

In particular, we have $$ \pr(\mathcal{R}_{2,k,D} \text{ is Radon with tolerance $\lfloor k/(4d + 4) \rfloor$} ) \geq 1/2.$$ 
\label{radontol}
\end{theorem}

Remark: Theorem~\ref{radontol} yields a weaker expected tolerance than Sober\'ons result, but the proof is shorter and more elementary.

For random allocations with more than two colors, we will use some developments on random allocation problems, including the following notation. If balls are thrown into $m$ urns uniformly and independently, let $N_m(n)$ equal the number of throws necessary to obtain at least $m$ balls in each urn. 

\begin{corollary}[Stochastic Tverberg for random allocation]~\label{probtverbergtolrand}
Suppose $D$ is  a probability distribution on $\mathbb{R}^d$ that is balanced about some point $\p \in \mathbb{R}^d$.

\begin{enumerate}

\item Then
$$ \pr(\mathcal{R}_{m,k,D} \text{ is Tverberg with tolerance $t$})
\geq \pr(N_n(m)\leq k) \left(1-2^{-\lfloor n / 2d \rfloor} \sum_{i = 1}^k \binom{\lfloor n / 2d \rfloor}{i}\right)^m.$$

\item For the case of Tverberg without tolerance, we also have 
$$ \pr(\mathcal{R}_{m,k,D} \text{ is Tverberg})
\geq \pr(N_n(m)\leq k)  \left(1-\left(2^{-n + 1} \sum_{k = 0}^{d-1} \binom{n-1}{k} \right)\right)^m.$$

\item Suppose $\mathcal{R}_{m,f(m),D}$, $m \in \mathbb{N}$ is a sequence of random partitioned point sets, where $n = f(m)$ depends on $m$. 

Then $\mathcal{R}_{m,f(m),D}$ is Tverberg with high probability if $f(m) \gg m \log_2(m) \ln(\ln(m))$.
\end{enumerate}

\end{corollary}

These results are improvements on Sober\'on's bound when the number of colors is large relative to the desired tolerance.

\section{Proofs of our stochastic results}~\label{bounds}

\begin{proof}[Proof of Theorem~\ref{discretesep}]
Let $M$ denote the minimal number of points perturbed among any perturbation that makes $(\X,y)$ separable, and $N$ denote the minimal number of points needing to be removed from  $(\X,y)$ to make $(\X,y)$ separable. Then $\text{PertSEP*}_0(\X,y)$ is equal to $M/n$, and the tolerance $t$ of $\X_1$, $\X_2$ is equal to $N$. It suffices to show that $M = N$. To see that $M \geq N$, note if $\x_1, \dots \x_M$ in $\X$ are moved so that the resulting set $(\X',y')$ is separable, then $(\X \setminus\{\x_1,\x_2,\dots,\x_N\},y \setminus \{y_1, \dots y_N\})$ is also separable. To see that $M \leq N$, suppose that $(\X \setminus\{\x_1,\x_2,\dots,\x_M\},y \setminus \{y_1, \dots y_M\})$ is separable by a hyperplane. Then moving $\x_1, \dots \x_M$ to the appropriate sides of the hyperplane determined by $h$, we can construct a separable dataset $(\X',\y')$, obtained from moving $M$ points from $(\X,\y)$.
\end{proof}

\begin{proof}[Proof of Theorem~\ref{pertsepmulti}]\sloppy
For fixed $m,d,$ and $\delta$, let $E_0(N)$ denote the event that a random allocation of $N$ points in $\R^d$ in $m$ colors has tolerance at least $(1-\delta)N/m$. By Sober\'on's theorem above, $E_0(N)$ asymptotically approaches one as $N$ goes to infinity. Now, for fixed $m$ and $\epsilon$, let $E_i(N)$ denote the event that a random allocation of $N$ points into $m$ colors has between $(1-\epsilon)N/m$ and $(1+\epsilon)N/m$ points of color $i$, where $i \in [m]$. By the law of large numbers, $E_i(N)$ approaches one as $N$ goes to infinity. As the $m+1$ events $E_j(N)$, where $j = \{0,1,\dots, m\}$, all have probability approaching one, the probability of the intersection of all these events also approaches one. This can be seen by applying the union bound to their complements. Thus there exists $N' \in \mathbb{N}$ such that the $E_j(N')$, where $j = \{0,1,\dots, m\}$, simultaneously occur with probability $(1-\epsilon_2)$. Therefore with probability $(1-\epsilon_2)$, each pair of colors has at most $(1+\epsilon)2N'/m$ points, and is a Radon partition of tolerance at least $(1-\delta)N'/m$ (the tolerance of each bi-partition is a priori bounded below by the tolerance of the $m$-partition). By theorem \ref{discretesep}, $\text{PertSEP*}_0$ of each pair is at least $(1-\delta)/2(1+\epsilon)$ with probability $(1-\epsilon_2)$. Since $\delta$ and $\epsilon$ were arbitrary, this completes the proof.
\end{proof}
\begin{proof}[Proof of the lower bound in Theorem~\ref{probtverberg}]
After a possible translation, can assume without loss of generality that $D$ is balanced about the origin. We will prove that$$ \left(1-\left(2^{-n + 1} \sum_{k = 0}^{d-1} \binom{n-1}{k} \right)\right)^m \leq \pr( \mathcal{P}_{m,n,D}\text{ is Tverberg} ) $$
by bounding from below the probability that the origin is a Tverberg point. We may assume without loss of generality that none of the randomly selected points are the origin. Furthermore we can radially project the points onto a sphere of radius smaller than the minimal norm of the projected points, since that will not affect whether the origin is a Tverberg point. After this projection, we may assume the points are uniformly sampled on a small sphere centered at the origin. The origin is then a Tverberg point as long as the points from each color contain the origin in their convex hull. This is equivalent to showing no color has all of its points contained in one hemisphere. For a fixed color, the probability of the $n$ points of that color being contained in one hemisphere was computed by Wagner and Welzl~\cite{wagnerwelzl} (generalizing the celebrated result of Wendel~\cite{wendel} addressing the case when $D$ is rotationally invariant about the origin) as 
\begin{equation}\left(2^{-n + 1} \sum_{k = 0}^{d-1} \binom{n-1}{k} \right).\end{equation}~\label{genwendel}
Using this to compute the probability that none of the $m$ color classes is contained in one hemisphere we obtain the desired bound above.

\end{proof}

\begin{proof}[Proof of the upper bound in Theorem~\ref{probtverberg}]
Again, we assume without loss of generality that $D$ is balanced about the origin.
We will first treat the case $d=1$, and then explain how to obtain the bound for arbitrary $d$.
To bound the probability of a Tverberg partition from above, we bound the probability of the complement below. 
 We let $E$ denote the event that the convex hulls have empty intersection. In dimension one, $E$ is contained in the event that there is at least one color class with all points less than zero, and at least one color class with all points greater than zero. Since we assume that the origin equipartitions $D$, we can rephrase this as the probability that among $m$ people each flipping $n$ fair coins, there is at least one person with all heads and at least one person with all tails.
In other words, denoting by $H$ and $T$ the events that at least one person gets all heads or tails respectively, we have $\pr(E) \geq \pr(H \cap T)$. 
We have $$\pr(H \cap T) = \pr(H) + \pr(T) - \pr(H\cup T) = \pr(H) + \pr(T) - (1-\pr((H\cup T)^c).$$
Since $\pr(H) = \pr(T) = (1-2^{-n})^m$ and $\pr((H\cup T)^c) = 1-(1-2^{-n+1})^m$, this yields
$$\pr(H \cap T) = 1 + (1-2^{-n + 1})^m  - 2(1-2^{-n})^m.$$
The probability of a Tverberg partition is thus bounded as follows $$\pr( \mathcal{P}_{m,n,D}\text{ is Tverberg} )  \leq 1-\pr(E) \leq 1 -\pr(H\cap T) =  2(1-2^{-n})^m - (1-2^{-n + 1})^m.$$

This proves the desired bound for dimension one. For higher dimensions, we note that if we let $p_i$,  denote the projection onto the $i$-th axis for $i \leq d$, we have that the signs of $p_1(x), \dots, p_d(x)$ are independent Bernoulli random variables with probability $1/2$ (as the hyperplane orthogonal to the $i$-th axis equipartitions $D$ by the assumption that $D$ is balanced about the origin).  Thus to have a Tverberg partition, we must have that no pair of the color classes are separated by the origin after projecting onto the $d$ coordinate axes. Since these $d$ events are independent, the probability of this happening is bounded as follows.

$$ \pr( \mathcal{P}_{m,n,D}\text{ is Tverberg} ) \leq 
\left( 2(1-2^{-n})^m - (1-2^{-n + 1})^m\right)^d.$$
\end{proof}

\begin{proof}[Proof of Theorem~\ref{TTT}]
We will show that $\mathcal{P}_{m,f(m),D}$ is Tverberg with high probability if $f(m) > \ln(m)/\ln(2)$. Fix an $\epsilon > 0$. We set $n = (1+ \epsilon)  \log_2(m)$ and apply the lower bound in Theorem~\ref{probtverberg} to deduce that $$ \pr(\mathcal{P}_{m,n,D} \text { is Tverberg }) \geq \left(1-\left(2^{-(1+ \epsilon) * \log_2(m) + 1} \sum_{k = 0}^{d-1} \binom{n-1}{k} \right)\right)^m $$

 $$ = \left(1-\left(2 m^{-(1 +\epsilon) } \sum_{k = 0}^{d-1} \binom{n-1}{k} \right)\right)^m . $$
 
Choosing a constant $K$ so that $K n^{d}\geq  2 \sum_{k = 0}^{d-1} \binom{n-1}{k}$, we have 

$$(1- Kn^dm^{-(1+\epsilon)})^m \leq \pr(\mathcal{P}_{m,n,D} \text { is Tverberg }) . $$

We will show that the limit as $m$ approaches infinity of the left hand side is bigger than $e^{-\delta}$ for any $\delta > 0$. Fix $\delta > 0$. As $n^d \sim O(\ln(m)^d)$, there exists an $M$ such that $Kn^dm^{-\epsilon} < \delta$ for all $m \geq M$. Consequently
$(1 - Kn^dm^{-(1+\epsilon)})^m > (1 - \delta m^{-1})^m$ for all $m \geq M$.
 Thus $$\lim_{m\rightarrow \infty}(1-Kn^dm^{-(1+\epsilon)})^m \geq \lim_{m \rightarrow \infty} (1-\delta m^{-1} )^m = e^{-\delta}. 
$$
Since $\delta$ was arbitrary, we see that the probability of a Tverberg partition tends to 1.

Now we show that $\mathcal{P}_{m,f(m),D}$ is not Tverberg with high probability if $f(m) < \log_2(m)$. As before, we fix an $\epsilon$ greater than zero apply the upper bound in Theorem~\ref{probtverberg} with $n = (1- \epsilon)  log_2(m)$ to obtain

$$\mathbb{P}( \mathcal{P}_{m,n,D} \text { is Tverberg })\leq 
\left( 2(1-m^{-1 + \epsilon})^m - (1-2m^{-1 + \epsilon})^m\right)^d.$$
For any $\gamma > 0$, when $m$ is large, both terms inside the parentheses are smaller than $(1- \gamma m^{-1})^m$. Since $\lim_{m \rightarrow \infty} (1-\gamma m^{-1} )^m = e^{-\gamma}$, the probability of a Tverberg partition converges to zero as $m$ approaches infinity.

\end{proof}

\begin{proof}[Proof of Theorem~\ref{probtverbergtol}]
Again, we assume without loss of generality that $D$ is balanced about $0$. 
Let $S$ denote the set of points of some fixed color. Then we assume that $|S| = n$, and we can partition $S$ into $\lfloor n/2d \rfloor $ subsets $S_1, \dots, S_{\lfloor n/2d \rfloor}$ with $S_i \geq 2d$ for each $i$. By Wagner and Welzl's result (Equation \ref{genwendel} above), for each $i$, $\conv(S_i)$ contains the origin with probability at least $1/2$. By independence, the probability that less than $t+1$ of the $S_i$ contain the origin is less than $2^{-\lfloor n/2d\rfloor}\sum_{i = 1}^t \binom{\lfloor n/2d \rfloor}{i}$. On the other hand, if at least $t +1$ of the $\conv(S_i)$ contain the origin, then by pigeonhole principle $\conv(S \setminus \{\x_1, \dots, \x_t\})$ contains the origin for any $\x_1, \dots \x_t \in S$.
Thus, with probability at least $1- 2^{-\lfloor n/2d\rfloor}\sum_{i = 1}^t \binom{\lfloor n/2d \rfloor}{i}$, we have that $\conv(S) \setminus \{\x_1, \dots, \x_t\})$ contains the origin. Since this probability is independent for each of the $m$ colors, the result follows. 
\end{proof}
Using a similar strategy combined with Cover's result, we give the proof of Theorem~\ref{stoctolradon} below.
\begin{proof}[Proof of Theorem~\ref{stoctolradon}]
Given $k$ points in $\R^d$ colored red and blue by random allocation, we arbitrarily partition them into $\lfloor k/(2d+2)\rfloor$ groups of size at least $2d+2$. By Cover's result, for each fixed group, the convex hulls (of the red and blue points) in that group intersect with probability at least $1/2$. For each of the $\lfloor k/(2d+2) \rfloor$ groups, we think of the event that the convex hulls in that group intersect as a ``success". Then the probability that at least $t+1$ groups have intersecting convex hulls is bounded below by the probability that a binomial process with $\lfloor k/(2d+2)\rfloor$ trials and success probability $1/2$ has at least $t+1$ total successes. Computing this binomial probability yields the theorem. (If at least $t+1$ groups have intersecting convex hulls, then removing at most $t$ points leaves at least one group with intersecting convex hulls. )
\end{proof}

\begin{proof}[Proof of Corollary~\ref{probtverbergtolrand}]
We split the proof according to the three respective parts of the statement.
\begin{enumerate}
\item The probability that a random allocation of $k$ points into $m$ colors is an $m$-Tverberg partition with tolerance $t$ is bounded below by the probability that a random allocation of $k$ points into $m$ colors has at least $n$ points per color, times the probability that an equipartition of $nm$ points into $m$ colors is Tverberg with tolerance $t$. The result for Tverberg with tolerance then follows from Theorem~\ref{probtverbergtol}.
\item The result for the special case of Tverberg without tolerance then follows the same reasoning as part (1), except using Theorem~\ref{probtverberg} in place of Theorem~\ref{probtverbergtol}. 
\item 
To show the asymptotic result, we use a result on urn models due to Erd\H{o}s and Renyi \cite{erdosrenyi} saying that $$\lim_{n\rightarrow\infty}\pr\left( \frac{N_m(n)}{n}< \log(n) + (m-1)\log(\log(n)) + x\right) = exp\left(-\frac{e^{-x}}{(m-1)!}\right).$$ This implies that for any $\epsilon > 0$ and sequence of $\log(\log(m))log_2(m)(1+ \epsilon)$ points allocated into $m$ urns, we have at least $log_2(m)(1 + \epsilon/2)$ points in each urn with high probability. Then we apply Theorem \ref{TTT}, which says that any equi-partition of a point set into $m$ colors and $log_2(m)(1 + \epsilon/2)$ points per color is Tverberg with high probability.
\end{enumerate} 
\end{proof}

\section{Acknowledgments} 
This work was partially supported by NSF grants DMS-1522158 and DMS-1818169. We are grateful to David Rolnick and 
Pablo Sober\'on for their comments.

\bibliographystyle{siamplain}
\bibliography{references__1}
\end{document}